\documentclass[12pt]{amsart}
\usepackage{amssymb}
\usepackage{amsmath}
\usepackage{amsfonts,latexsym,amstext}

\usepackage [T1]{fontenc}
\usepackage [latin1]{inputenc}
\usepackage{color}

\newtheorem{theorem}{Theorem}[section]
\newtheorem{definition}[theorem]{Definition}

\newtheorem{example}[theorem]{Example}
\newtheorem{lemma}[theorem]{Lemma}
\newtheorem{proposition}[theorem]{Proposition}
\newtheorem{corollary}[theorem]{Corollary}
\newtheorem{remark}[theorem]{Remark}

\def\u{\mathfrak{A}}

\def\p{\mathcal{P}}
\def\v{\overset\vee}
\def\l{\mathcal{L}}

\def\fijn{P_{a^{n-1}}}

\newcommand{\tri[2]}{\|\hspace{-0.04cm}|#2\|\hspace{-0.04cm}|}

\begin{document}

\title{Every  Banach ideal of polynomials is compatible with an operator ideal
}

\author{Daniel Carando}

\author{Ver\'{o}nica Dimant}
\author{Santiago Muro}

\thanks{Partially supported by ANPCyT PICT 05 17-33042. The first and third authors were also partially supported by UBACyT Grant X038 and ANPCyT PICT 06
00587.}

\address{Departamento de Matem\'{a}tica - Pab I,
Facultad de Cs. Exactas y Naturales, Universidad de Buenos Aires,
(1428) Buenos Aires, Argentina and CONICET.}
\email{dcarando@dm.uba.ar}

\address{Departamento de Matem\'{a}tica, Universidad de San
Andr\'{e}s, Vito Dumas 284, (B1644BID) Victoria, Buenos Aires,
Argentina and CONICET.} \email{vero@udesa.edu.ar}

\address{Departamento de Matem\'{a}tica - Pab I,
Facultad de Cs. Exactas y Naturales, Universidad de Buenos Aires,
(1428) Buenos Aires, Argentina and CONICET.} \email{smuro@dm.uba.ar}

\begin{abstract}
We show that for each Banach ideal of homogeneous polynomials,
there exists a (necessarily unique) Banach operator ideal compatible
with it. Analogously, we prove that any ideal of $n$-homogeneous polynomials
belongs to a  coherent sequence of ideals of
$k$-homogeneous polynomials.
\end{abstract}


\maketitle

\section{Introduction}

Most examples of polynomial ideals were inspired in an ideal of
operators. This is the case, for example, of the ideals of nuclear,
integral, compact, multiple $r$-summing or $r$-dominated polynomials.
However, the extension of a linear operator ideal to higher
degrees is not always obvious. For example, many extensions of the
ideal of absolutely $r$-summing operators have been developed,
among them, the absolutely, the multiple and the strongly
$r$-summing polynomials and the $r$-dominated polynomials.

The
question tackled in this article is whether every ideal of
polynomials is an extension of an ideal of operators. This is not a
precise question unless we settle what is understood by an
``extension of an ideal of operators''.

In \cite{CarDimMur09} we defined and studied
the concept of a homogeneous polynomial ideal being
\textit{compatible} with an operator ideal. This notion aims to clarify the relationship
between an operator ideal $\u$ and a polynomial ideal defined ``in
the spirit of $\u$''.
Compatibility is related with the natural operations of
fixing variables and multiplying by linear functionals.
We proved in \cite[Proposition 1.6]{CarDimMur09} that each ideal of polynomials can be compatible with, at most, one ideal of operators. On the other hand, an operator ideal is always compatible with several different polynomial ideals.
In this article we complete these results by showing that a Banch ideal of polynomials is always compatible with a (necessarily unique) ideal of operators.

The concept of coherent sequence of polynomial ideals was also
introduced in \cite{CarDimMur09}, in the spirit of Nachbin's
holomorphy types~\cite{Nac69}. It is natural to ask
if every ideal of homogeneous polynomials belongs to a coherent
sequence. We answer this question affirmatively: for a fixed ideal of $n$-homogeneous polynomials $\u_n$,
there exist polynomial ideals $\u_1,\dots,\u_{n-1},\u_{n+1},\dots$ such that, together with $\u_n$, form a coherent sequence. That is, a sequence of polynomial ideals of
degree $k$ ($1\le k\le n-1$) all ``associated'' to $\u_n$.

\bigskip

Throughout this paper $E$, $F$ and $G$ will be complex Banach
spaces. We denote by $\l(E,F)$ the Banach space of all continuous
linear operators from $E$ to $F$ and by $\p^n(E,F)$ the Banach space
of all continuous $n$-homogeneous polynomials from $E$ to $F$. If
$P\in \p^n(E,F)$, there exists a unique symmetric $n$-linear mapping
$\v P\colon\underbrace{E\times\cdots\times E}_n\to F$ such that
$$P(x)=\v P(x,\dots,x).$$

We define $P_{a^k}\in \p^{n-k}(E,F)$ by $$ P_{a^k}(x)=\v
P(a^k,x^{n-k}).$$ For $k=1$, we write $P_a$ instead of $P_{a^1}$.

Let us recall the definition of polynomial ideals
\cite{Flo01,Flo02,FloHun02}. A \textbf{normed ideal of continuous
$n$-homo\-geneous polynomials} is a pair
$(\mathfrak{A}_n,\|\cdot\|_{\mathfrak A_n})$ such that:
\begin{enumerate}
\item[(i)] $\mathfrak{A}_n(E,F)=\mathfrak A_n\cap \mathcal
P^n(E,F)$ is a linear subspace of $\p^n(E,F)$ and
$\|\cdot\|_{\u_n(E,F)}$ is a norm on it.

\item[(ii)] If $T\in \l (E_1,E)$, $P \in \u_n(E,F)$ and $S\in \l
(F,F_1)$, then $S\circ P\circ T\in \u_n(E_1,F_1)$ and $$ \| S\circ
P\circ T\|_{\u_n(E_1,F_1)}\le \| S\| \|P\|_{\u_n(E,F)} \| T\|^n$$

\item[(iii)] $z\mapsto z^n$ belongs to $\u_n(\mathbb C,\mathbb C)$
and has norm 1.
\end{enumerate}

In \cite{CarDimMur09} we defined and studied
the concept of a polynomial ideal being
compatible with an operator ideal. Here we present more results
about this topic. We first recall the definition:

\begin{definition}\label{deficompatible}
Let $\u$ be a normed ideal of linear operators. We say that the
normed ideal of $n$-homogeneous polynomials $\u_n$ is
\textbf{compatible with $\u$} if there exist positive constants
$A$ and $B$ such that for every Banach spaces $E$ and $F$, the
following conditions
hold: \\
(i) For each $P\in \u_n(E,F)$ and $a\in E$, $\fijn$ belongs to
$\u(E;F)$ and
$$\|\fijn\|_{\u(E,F)} \le A
\|P\|_{\u_n(E,F)} \|a\|^{n-1}$$
(ii) For each $T\in \u(E,F)$ and $\gamma\in E'$, $\gamma^{n-1}T$
belongs to $\u_n(E,F)$ and
$$\|\gamma^{n-1}T\|_{\u_n(E,F)}\le B \|\gamma\|^{n-1}
\|T\|_{\u(E,F)}$$
\end{definition}

All the examples of ideals of polynomials mentioned above, with the exception of the ideal of absolutely summing polynomials, are compatible with the respective ideal of
operators (see \cite[Section 1]{CarDimMur09} and Example \ref{abs sum compatibles con todos}).

\begin{remark}\label{no hace falta la norma}\rm
Although the definition of compatibility involves constants which relate the norms of the operators and the homogeneous polynomials, a simple application of the closed graph theorem shows that when the ideals are complete those constants automatically exist (see \cite{Mur10} for details). This means that if we can define the operations of fixing variables and multiplying by functionals, then they are uniformly (in the Banach spaces $E,F$) bounded.

Even though it is not necessary to obtain the constants $A$ and $B$
to show that two Banach ideals are compatible, we will also seek
``good'' constants mostly for two reasons: the first one is that
they provide a bound for the norm of the derivatives of homogeneous
polynomials in different ideals and the second is that this kind of
bounds allow us to define holomorphic mappings associated to
sequences of ideals (see \cite{CarDimMur}).
\end{remark}

\section{Existence of a compatible operator ideal}\label{seccion existe compatible}

In \cite{CarDimMur09} it was shown that for any given operator ideal there is more than one ideal of polynomials compatible with it, for example the ideals of 2-dominated and multiple 2-summing 2-homogeneous polynomials are both compatible with the ideal of absolutely 2-summing operators. Also, it was proved there that there exist at most one operator ideal compatible with a given polynomial ideal. On the other hand, not every
polynomial ideal is compatible with the commonly associated operator ideal  (e.g. the ideal absolutely 1-summing polynomials is not compatible with the ideal of absolutely 1-summing operators \cite[Example 1.15]{CarDimMur09}).

 It is natural to ask wether every polynomial ideal must have a (necessarily unique) compatible operator ideal or not.
 The following result answers this question affirmatively.
\begin{theorem} \label{existcompat}
Let $\u_n$ be a Banach ideal of $n$-homogeneous polynomials. Then
there exists a unique Banach ideal of operators $\u$ compatible
with $\u_n$. This operator ideal can be normed to obtain compatibility constants
$1\le A,B\le e$.
\end{theorem}

The proof will be given in several steps. First, we need the following Lemma:
\begin{lemma}\label{rrr}
Let $\u_n$ a normed ideal of $n$-homogeneous polynomials and
$P\in\u_n(E,F)$. If $T, S \in \mathcal L(G,E)$, then the
$n$-homogeneous polynomial $Q(\cdot)=\v
P(T(\cdot),\cdots,T(\cdot),S(\cdot))$ belongs to $\u_n(G,F)$.
Moreover, $\|Q\|_{\u_n(G,F)} \le e\|T\|_{\mathcal L(G,E)}^{n-1}\|S\|_{\mathcal L(G,E)}\|P\|_{\u_n(E,F)}$.

In particular, if $S \in \mathcal L(G,E)$, $\gamma_1,\dots,\gamma_k \in E'$, $k<n$ and $a\in E$, then
$\gamma_1\dots\gamma_k(P_{a^k}\circ S)\in\u_n(G,F)$; and if $\gamma\in E'$ then:
\begin{enumerate}
\item[($a$)] $\gamma^{n-1}(P_{a^{n-1}}\circ S)\in\u_n(G,F)$ with
$$\|\gamma^{n-1} (P_{a^{n-1}}\circ S)\|_{\u_n(G,F)}\le
e\|\gamma\|^{n-1}\|a\|^{n-1}\|P\|_{\u_n(E,F)}\|S\|_{\l(G,E)}.$$

\item[($b$)] $\gamma (P_{a}\circ S)\in\u_n(E,F)$ with $$\|\gamma
(P_{a}\circ S)\|_{\u_n(E,F)}\le e\|\gamma\|\|a\|\|P\|_{\u_n(E,F)}\|S\|^{n-1}.$$
\end{enumerate}
\end{lemma}

\begin{proof} 
As in \cite[Corollary 1.8]{CarDimMur09}, we can write
$Q$ in the following useful way:
$$Q(x)=\frac{1}{n^2} \frac{1}{(n-1)^{n-1}} \sum_{j=0}^{n-1}r^j P\left((n-1){r^j}T(x)+S(x)\right),$$
where $r$ is a primitive $n$th root of unity. Thus, defining, for each $0\le j\le n-1$,  the
linear operator
$$S_j(x)=(n-1){r^j}T(x)+S(x),$$ we have that
$$Q=\frac{1}{n^2}\frac{1}{(n-1)^{n-1}}\sum_{j=0}^{n-1}r^j \big(P\circ
S_j\big). $$ Therefore, $Q$ belongs to $\u_n(G,F)$.

For the estimation of the norm, it is enough to consider the case
$\|S\|=\|T\|=1$. Since $\|S_j\|\le n$, for every $j=0,\dots, n-1$,
we obtain
$$\|Q\|_{\u_n(G,F)}\le \frac{1}{n^2}\frac{1}{(n-1)^{n-1}} n
\|P\|_{\u_n(G,F)} n^n=
\frac{n^{n-1}}{(n-1)^{n-1}}\|P\|_{\u_n(G,F)} \le
e\|P\|_{\u_n(G,F)}.$$

For the particular cases, just note that $\gamma^{n-1}
\left(P_{a^{n-1}}\circ S\right)(x)=\v P(\gamma(x)a,\cdots,\gamma(x)a,S(x))$, and $\gamma (P_{a}\circ S)(x)=\v P(\gamma(x)a,S(x),\cdots,S(x))$.
\end{proof}

As a consequence of this lemma we obtain the following.

\begin{lemma} \label{siunotodos}
Let $\u_n$ be an ideal of $n$-homogeneous polynomials, let $T\in
\mathcal L(E,F)$ and fix a nonzero $\gamma_0 \in E'$. Then
$\gamma_0^{n-1}T\in \u_n(E,F)$ if and only if $\gamma^{n-1}T\in
\u_n(E,F)$ for every $\gamma\in E'$.
\end{lemma}
\begin{proof}
Pick $a\in E$ such that $\gamma_0(a)\ne 0$. By Lemma \ref{rrr},
$\gamma^{n-1}\left(\gamma_0^{n-1}T\right)_{a^{n-1}}\in \u_n(E,F)$.
We have
$$
\gamma^{n-1}\left(\gamma_0^{n-1}T\right)_{a^{n-1}}(x)=\frac{\gamma(x)^{n-1}}{n}
\Big(\gamma_0(a)^{n-1}T(x)+(n-1)\gamma_0(a)^{n-2}\gamma_0(x)T(a)\Big).
$$
Therefore
$$
\left(\gamma^{n-1}T\right)(\cdot)=
\frac{n}{\gamma_0(a)^{n-1}}\left(\gamma^{n-1}(\cdot)(\gamma_0^{n-1}T)_{a^{n-1}}(\cdot)
-\frac{n-1}{n}\gamma^{n-1}(\cdot)\gamma_0(\cdot)\gamma_0(a)^{n-2}T(a)\right),
$$ and then $\gamma^{n-1}T$ belongs to $\u_n(E,F)$.
\end{proof}

Now we can define, for a fixed polynomial ideal $\u_n$, an
operator ideal $\u$, and a complete norm on it. This norm also has some interesting properties that we
present in the following proposition.

\begin{proposition}\label{prop quasi ideal}
Let $\u_n$ be an ideal of $n$-homogeneous polynomials. Define, for
each pair of Banach spaces $E$ and $F$,
$$
\u(E,F) = \left\{ T\in \mathcal L(E,F) /\; \gamma^{n-1}T\in
\u_n(E,F) \textrm{ for all } \gamma\in E' \right\},
$$
with $\tri{T}_{\u(E,F)}=\sup_{\gamma\in
S_{E'}}\|\gamma^{n-1}T\|_{\u_n(E,F)}$. Then
\begin{enumerate}
\item[($a$)] $\u$ is an ideal of operators and $\u(E,F) = \left\{
P_{a^{n-1}}\in \mathcal L(E,F) /\; P\in \u_n(E,F), a\in E
\right\}$.

\item[($b$)] $\tri{\cdot}_{\u(E,F)}$ is a norm on $\u(E,F)$ and
verifies $$\tri{T}_{\u(E,F)}\ge\|T\|_{\mathcal{L}(E,F)},\quad
\textrm{for every } T\in\u(E,F).$$ Moreover,
$\Big(\u(E,F),\tri\cdot_{\u(E,F)}\Big)$ is a Banach space.

\item[($c$)] $\tri{S\circ T}_{\u(E,F_1)}\le \|S\|_{\mathcal
L(F,F_1)}\tri{T}_{\u(E,F)}$ for every $S\in\mathcal L(F,F_1)$ and
$T\in\u(E,F)$.

\item[($d$)] If $E_0$ is a subspace of $E$ with norm 1 inclusion
$i:E_0\hookrightarrow E$, then $$\tri{T\circ i}_{\u(E_0,F)}\le
\tri{T}_{\u(E,F)},\quad \textrm{for all }T\in\u(E,F).$$
\end{enumerate}
\end{proposition}

\begin{proof}

($a$) Clearly the sum and multiplication by scalars of members of
$\u$ is again in $\u$. So, to prove that $\u$ is an ideal of
operators, we have to show that it behaves well with compositions.

Consider $T\in \u(E,F)$, $R\in \mathcal L(E_1,E)$ and $S\in
\mathcal L(F,F_1)$. Let us prove that $S\circ T\circ R\in
\u(E_1,F_1)$. Let $\gamma\in E'$ such that $\gamma \circ R \ne 0$.
Then $\gamma^{n-1}T \in \u_n(E,F)$ and $\eta=\gamma \circ R\in
E_1'$. By Lemma \ref{siunotodos}, it suffices to show that
$\eta^{n-1}\big(S\circ T\circ R\big)\in \u_n(E_1,F_1)$. This
follows from the equalities:
$$
\Big(\eta^{n-1}\big(S\circ T\circ
R\big)\Big)(x)=\gamma^{n-1}\big(R(x)\big)S\big(T\big(R(x)\big)\big)
= \big(S\circ (\gamma^{n-1}T)\circ R\big)(x).
$$
Therefore $\u$ is an ideal of operators.

To prove the equivalent definition of $\u$, suppose
$T=P_{a^{n-1}}$ with $P\in \u_n(E,F)$ and $a\in E$. Then by Lemma
\ref{rrr}, $\gamma^{n-1}T$ belongs to $\u_n(E,F)$, for all
$\gamma\in E'$, and thus $T\in\u(E,F)$.

Conversely, if $T\in\u(E,F)$ then $\gamma^{n-1}T\in\u_n(E,F)$ for
every $\gamma\in E'$. Let $a\in E$ such that $\gamma(a)=1$, then
$P=n\gamma^{n-1}T-(n-1)T(a)\gamma^n$ is in $\u_n(E,F)$ and
$P_{a^{n-1}}=T$.

($b$) It is straightforward to prove that we defined a norm.

Let $T\in\u(E,F)$, take $x\in S_E$ such that
$\|T(x)\|>\|T\|_{\mathcal{L}(E,F)}-\varepsilon$ and $\gamma\in
S_{E'}$ such that $|\gamma(x)|=1$.  Then,
$$\tri{T}_{\u(E,F)}\ge\|\gamma^{n-1}
T\|_{\u_n(E,F)}\ge\|\gamma^{n-1}
T\|_{\p^n(E,F)}\ge\|\gamma(x)^{n-1}T(x)\|>\|T\|_{\mathcal{L}(E,F)}-\varepsilon.$$
Since this is true for every $\varepsilon>0$, we have that
$\tri{T}_{\u(E,F)}\ge\|T\|_{\mathcal L(E,F)}$.

Let us see that $\Big(\u(E,F),\tri\cdot_{\u(E,F)}\Big)$ is
complete. Suppose $\sum_{k\in\mathbb N}\tri{T_k}_{\u(E,F)}$ is
convergent. Then $\sum_{k\in\mathbb N}\|T_k\|_{\mathcal L(E,F)}$
is convergent. Therefore there exists $T\in\mathcal L(E,F)$ such
that $\sum_{k}T_k\to T$ in $\mathcal L(E,F)$.

For each $\gamma\in S_{E'}$, we know that $\gamma^{n-1}
T_k\in\u_n(E,F)$ and
$\|\gamma^{n-1}T_k\|_{\u_n(E,F)}\le\tri{T_k}_{\u(E,F)}$. Thus,
$\sum_{k}\gamma^{n-1} T_k$ converges in $\u_n(E,F)$ and its limit
has to be $\gamma^{n-1} T$. Therefore, $T$ belongs to $\u(E,F)$.
Moreover, since
$$
\sup_{\gamma\in S_{E'}}\left\|\gamma^{n-1}\sum_{k\ge
N}T_k\right\|_{\u_n(E,F)} \le \sup_{\gamma\in S_{E'}}\sum_{k\ge
N}\left\|\gamma^{n-1}T_k\right\|_{\u_n(E,F)}\le\sum_{k\ge
N}\tri{T_k}_{\u(E,F)}\to 0,
$$
as $N\to\infty$, we have that $\sum_{k}T_k\to T$ in
$\Big(\u(E,F),\tri\cdot_{\u(E,F)}\Big)$.

($c$) For every $S\in\mathcal L(F,F_1)$ and $T\in\u(E,F)$, we
have:

\begin{eqnarray*}
\tri{S\circ T}_{\u(E,F_1)}&=&\sup_{\gamma\in
S_{E'}}\|\gamma^{n-1}S\circ T\|_{\u_n(E,F_1)}= \sup_{\gamma\in
S_{E'}}\|S\circ(\gamma^{n-1}T)\|_{\u_n(E,F_1)}\\
&\le & \|S\|_{\mathcal L(F,F_1)}\sup_{\gamma\in
S_{E'}}\|\gamma^{n-1}T\|_{\u_n(E,F)}=\|S\|_{\mathcal
L(F,F_1)}\tri{T}_{\u(E,F)}.
\end{eqnarray*}

($d$) Let $T\in\u(E,F)$ and $\gamma\in E_0'$. Consider
$\tilde\gamma\in E'$ a Hahn-Banach extension of $\gamma$
preserving its norm. Then
$$
\Big\|\gamma^{n-1} (T\circ
i)\Big\|_{\u_n(E_0,F)}=\Big\|(\tilde\gamma\circ i)^{n-1} (T\circ
i)\Big\|_{\u_n(E_0,F)} =\big\|(\tilde\gamma^{n-1} T)\circ
i\big\|_{\u_n(E_0,F)}\le \|\tilde\gamma^{n-1} T\|_{\u_n(E,F)}.
$$
Taking supremum we have that
$$
\tri{T\circ i}_{\u(E_0,F)}\le \tri{T}_{\u(E,F)}.
$$
\end{proof}

The following proposition shows that the norm defined on $\u$ is
``almost ideal'', in the sense that satisfies the ideal condition up
to a constant.

\begin{proposition}\label{prop quasi ideal 2}
The norm $\tri\cdot_{\u}$ defined on Proposition \ref{prop quasi
ideal} verifies the ``almost ideal'' property: for Banach spaces
$E$ and $F$, there exists a constant $c>0$ such that, for all
Banach spaces $E_1,\ F_1$ and all operators
 $R\in \l (E_1,E)$, $T \in \u(E,F)$ and $S\in \l
(F,F_1)$, it follows that $$ \tri{ S\circ T\circ
R}_{\u(E_1,F_1)}\le c\| S\|_{\mathcal{L}(F,F_1)} \tri{T}_{\u(E,F)}
\| R\|_{\l (E_1,E)}.$$
\end{proposition}
\begin{proof}
By Proposition \ref{prop quasi ideal} ($c$), we have that
$$
\tri{ S\circ T\circ R}_{\u(E_1,F_1)}\le \| S\|_{\mathcal{L}(F,F_1)}
\tri{T\circ R}_{\u(E_1,F)}.
$$

For a fixed Banach space $E_1$ and a fixed operator $R\in\l
(E_1,E)$, consider
$$
\begin{array}{ccc}
 \left(\u(E,F),\tri\cdot_{\u(E,F)}\right) & \to  & \left(\u(E_1,F),\tri\cdot_{\u(E_1,F)}\right) \\
  T & \mapsto & T\circ R
\end{array}
$$
An application of the Closed Graph Theorem gives the existence of
a constant $c_{E_1,R}>0$ such that
$$
\tri{T\circ R}_{\u(E_1,F)}\le c_{E_1,R}\tri T_{\u(E,F)}.
$$
If we apply again the Closed Graph Theorem for
$$
\begin{array}{ccc}
 \l(E_1,E) & \to  & \l(\u(E,F),\u(E_1,F)) \\
  R & \mapsto & \theta_R(T)=T\circ R,
\end{array}
$$
we obtain that there is a constant $c_{E_1}>0$ such that
\begin{equation}\label{grafico cerrado}
\tri{T\circ R}_{\u(E_1,F)}\le c_{E_1}\tri
T_{\u(E,F)}\|R\|_{\l(E_1,E)}.
\end{equation}

Now suppose that the result is not true. Then there are Banach
spaces $E_k$, and $R_k\in\l(E_k,E)$, $\|R_k\|_{\l(E_k,E)}=1$, for
all $k\in\mathbb N$, such that
$$
\tri{T\circ R_k}_{\u(E_k,F)}>k.
$$
Let $E_0=\bigoplus_{k\in \mathbb N}E_k$, and $\tilde
R_k\in\l(E_0,E)$, $\tilde R_k=R_k\circ \pi_k$, where $\pi_k:E_0\to
E_k$ is the (norm one) projection. Denote by
$i_k:E_k\hookrightarrow E_0$ the (norm one) inclusion. So we have

\begin{eqnarray*}
k&<&\tri{T\circ R_k}_{\u(E_k,F)}=\tri{T\circ R_k\circ\pi_k\circ
i_k}_{\u(E_k,F)}\\
&=&\tri{T\circ \tilde R_k\circ i_k}_{\u(E_k,F)}\le \tri{T\circ
\tilde R_k}_{\u(E_0,F)},
\end{eqnarray*}
the last inequality following from Proposition \ref{prop quasi
ideal}($d$). Also, by  (\ref{grafico cerrado}),
$$
\tri{T\circ \tilde R_k}_{\u(E_0,F)}\le c_{E_0}\tri
T_{\u(E,F)}\|\tilde R_k\|_{\l(E_0,E)}\le c_{E_0}\tri T_{\u(E,F)},
$$
which leads to a contradiction.
\end{proof}

Now we present a result that shows how to convert an ``almost
ideal'' norm into an ideal norm.

\begin{proposition}\label{ideal norm}
Let $\u$ be an operator ideal with norm $\tri\cdot_{\u}$ that
verifies the ``almost ideal'' property. Then we can define an
equivalent norm $\|\cdot\|_{\u}$ which is an ideal norm on $\u$.
\end{proposition}
\begin{proof}
We first define a norm $\|\cdot\|'_{\u}$ in the following way. For
$T\in\u(E,F)$, let
$$\|T\|'_{\u(E,F)}=\sup\{\tri{S\circ T\circ R}_{\u(E_1,F_1)}:\
E_1,F_1 \textrm{ Banach spaces,
}\|S\|_{\l(F,F_1)}=\|R\|_{\l(E_1,E)}=1\}.
$$
It is easy to see that $\|\cdot\|'_{\u}$ is a norm on $\u$
equivalent to $\tri\cdot_{\u}$. Also, it is clear that verifies
the ideal property:
$$
\|S\circ T\circ R\|'_{\u(E_1,F_1)}\leq \|S\|_{\l(F,F_1)}
\|T\|'_{\u(E,F)}\|R\|_{\l(E_1,E)}.
$$
Last, if $\kappa=\|id_{\mathbb{C}}\|'_{\u(\mathbb{C},\mathbb{C})}$
then the norm $\|\cdot\|_{\u}$ defined by
$$
\|T\|_{\u(E,F)}=\frac{1}{\kappa}\|T\|'_{\u(E,F)}
$$
is an ideal norm equivalent to $\tri\cdot_{\u}$.
\end{proof}

\begin{remark}\label{norma prima simpli}\rm
When applying the previous proposition to our
context (that is, $\u_n$ a polynomial ideal and
$(\u,\tri\cdot_{\u})$ as in Proposition \ref{prop quasi ideal}),
using Proposition \ref{prop quasi ideal} $(ii)$, we can simplify
the definition of $\|\cdot\|'_{\u}$:
$$\|T\|'_{\u(E,F)}=\sup\{\tri{ T\circ R}_{\u(E_1,F)}:\
E_1 \textrm{ Banach space, }\|R\|_{\l(E_1,E)}=1\}.
$$
Then considering
$$\|T\|_{\u(E,F)}=\frac{\|T\|'_{\u(E,F)}}{\|id_{\mathbb{C}}\|'_{\u(\mathbb{C},\mathbb{C})}}
$$
we obtain an ideal norm on $\u$ equivalent to $\tri\cdot_{\u}$.
Moreover,
\begin{eqnarray*}
\kappa &=&\|z\mapsto z\|'_{\u(\mathbb{C},\mathbb{C})} =\sup\{\tri{(z\mapsto z)\circ \varphi}_{\u(E_1,\mathbb C)}:\
E_1 \textrm{ Banach space, }\varphi\in S_{E_1'}\}\\ &=&\sup\{\tri{\varphi}_{\u(E_1,\mathbb C)}:\
E_1 \textrm{ Banach space, }\varphi\in S_{E_1'}\}\\ & =&\sup\{\|\gamma^{n-1}\varphi\|_{\u_{n}(E_1,\mathbb C)}:\
E_1 \textrm{ Banach space, }\varphi,\gamma\in S_{E_1'}\}.
\end{eqnarray*}
 Thus by \cite[Corollary 1.8]{CarDimMur09}, we have that
$1\le\kappa\le e$.\qed
\end{remark}

\smallskip
We now can prove the existence, for any polynomial ideal, of a compatible operator ideal:

\begin{proof}(of Theorem \ref{existcompat})
Consider the normed ideal $(\u,\|\cdot\|_{\u})$, with $$\u(E,F) =
\left\{ T\in \mathcal L(E,F) /\; \gamma^{n-1}T\in \u_n(E,F)
\textrm{ for all } \gamma\in E' \right\}$$ and $\|\cdot\|_{\u}$
given by Remark \ref{norma prima simpli} $(ii)$. By the
equivalence with $\tri\cdot_{\u}$ and Proposition \ref{prop quasi
ideal} $(b)$, for each $E$ and $F$ Banach,
$\big(\u(E,F),\|\cdot\|_{\u(E,F)}\big)$ is a Banach space.

Let us check that $\u_n$ is compatible with $\u$.

It is clear, by
definition, that if $T\in \u(E,F)$ and $\gamma \in E'$ then
$\gamma^{n-1}T\in \u_n(E,F)$. On the other hand take $P\in \u_n(E,F)$
and $a\in E$. By Proposition \ref{prop quasi ideal} ($a$),
$P_{a^{n-1}}$ belongs to $\u(E,F)$. By Remark~\ref{no hace falta la norma} we conclude that $\u_n$ is compatible with $\u$. We can moreover estimate the constants of compatibility. For the first one, by Lemma \ref{rrr} ($a$),
$$
\|P_{a^{n-1}}\|_{\u(E,F)}=\frac1{\kappa}\sup_{\overset{E_1\;Banach}{R\in
S_{\l(E_1,E)}}}\sup_{\|\gamma\|=1}\Big\|\gamma^{n-1}(P_{a^{n-1}}\circ
R)\Big\|_{\u_n(E_1,F)} \le
\frac{e}{\kappa}\|a\|^{n-1}\|P\|_{\u_n(E,F)}.
$$
For the other constant we have,
$$
\|\gamma^{n-1}T\|_{\u_n(E,F)} =\|\gamma\|^{n-1}
\left\|\frac{\gamma^{n-1}}{\|\gamma\|^{n-1}}T\right\|_{\u_n(E,F)}
\le \|\gamma\|^{n-1}\tri{T}_{\u(E,F)}\le
\kappa\|\gamma\|^{n-1}\|T\|_{\u(E,F)}.
$$

The fact that $\u$ is the only ideal of operators compatible with $\u_n$ follows from \cite[Proposition 1.6]{CarDimMur09}.
\end{proof}

 We have proved that every polynomial Banach ideal is compatible with a unique Banach operator ideal. On the other hand, \cite[Example 1.15]{CarDimMur09} shows that the ideal of absolutely 1-summing polynomials is not compatible with the ideal of absolutely 1-summing operators. Then the question that comes up now is which is the ideal of linear operators which is compatible with the absolutely $1$-summing polynomials.

As the following example shows, the unique compatible operator ideal
may be far from ``natural''. Note,
however, that this unnatural compatibility has some interesting
consequences.

\begin{example}\label{abs sum compatibles con todos}
The ideal $\Pi_{p}^n$ of absolutely-$p$-summing $n$-homogeneous
polynomials is compatible with $\mathcal{L}$, the ideal of
continuous linear operators, with constants $A=e$ and $B=1$.
\end{example}
\begin{proof}
Obviously, for $P\in\Pi_p^n(E,F)$ and $a\in E$, the operator
$P_{a^{n-1}}$ belongs to $\mathcal{L}(E,F)$ and
$$
\|\fijn\|_{\mathcal{L}(E,F)} \le e \|P\|_{\mathcal{P}^n(E,F)}
\|a\|^{n-1}\le e \|P\|_{\Pi_p^n(E,F)} \|a\|^{n-1}.
$$

\noindent For the other condition, let $T\in\mathcal{L}(E,F)$ and
$\gamma\in E'$, then, for all $x_1,\dots,x_m\in E$,
\begin{eqnarray*}
\Big(\sum_{j=1}^m
\big\|(\gamma^{n-1}T)(x_j)\big\|^p\Big)^\frac1{p}&\le &
\|\gamma\|\left(\sum_{j=1}^m
\Big(\frac{|\gamma(x_j)|}{\|\gamma\|}\|\gamma\|^{n-2}
\|T\|\|x_j\|^{n-1}\Big)^p\right)^\frac1{p} \\
&\le&
\|\gamma\|^{n-1}\|T\|\left(\sum_{j=1}^m
\Big(\frac{|\gamma(x_j)|}{\|\gamma\|}\Big)^p\right)^\frac1{p} \Big(\max_{1\le j\le m}\|x_j\|\Big)^{n-1}\\
&\le& \|\gamma\|^{n-1}\|T\|\sup_{x'\in B_{E'}} \left(\sum_{j=1}^m
|x'(x_j)|^p\right)^{\frac{n}{p}} = \|\gamma\|^{n-1}\|T\|\omega_p\big((x_j)_{j=1}^m\big)^n.
\end{eqnarray*}

Thus, $\gamma^{n-1}T$ is absolutely $p$-summing and
$$
\|\gamma^{n-1}T\|_{\Pi_p^n(E,F)}\le \|T\|_{\mathcal{L}(E,F)}
\|\gamma\|^{n-1}.
$$
\end{proof}

\begin{corollary}
Suppose that $\Pi_p^n(E,F)\subset \u_n(E,F)$ and that $\u_n$ is compatible with $\u_1$. Then $\u_1(E,F)=\mathcal L(E,F)$.
\end{corollary}
\begin{proof}
This is just a special case of \cite[Proposition 1.6]{CarDimMur09}.
\end{proof}

It is well known that every absolutely summing operator is weakly compact (see for example \cite[Theorem 2.17]{DieJarTon95}). In \cite{Bot02} it was shown that not every dominated polynomial is weakly compact by exhibiting an example of a polynomial from $\ell_1$ to $\ell_1$. We now show how the concept of compatible ideals can be easily applied to prove that not every absolutely $p$-summing homogeneous polynomial is weakly compact.

\begin{corollary}\label{absolutamente sumantes y debil compactos}
$E$ is reflexive if and only if, for some $n\ge 2$, every absolutely $p$-summing $n$-homogeneous polynomial from $E$ to $E$ is weakly compact.
\end{corollary}
\begin{proof}
It easy to prove that
the ideal of weakly compact homogeneous polynomials, $\mathcal P^n_{WK}$, is compatible with
the ideal of weakly compact operators, $\mathcal L_{WK}$. 
Suppose that $\Pi_p^n(E,E)\subset\p_{WK}^n(E,E)$.
Then, by the previous corollary, we have that $\mathcal{L}(E,E)=\mathcal L_{WK}(E,E)$ and thus $E$ must be reflexive.

Conversely, if $E$ is reflexive, every homogeneous polynomial from $E$ to $E$ is weakly compact.
\end{proof}
\medskip

Analogously we can prove that if every absolutely $p$-summing
$n$-homogeneous polynomial from $E$ to $F$ is weakly compact (for
some $n\ge 2$), then every linear operator
from $E$ to $F$ is weakly compact.

\section{Coherent sequences of polynomial ideals}

In \cite{CarDimMur09} we also defined coherent sequences of polynomial ideals:

\begin{definition}\label{deficoherente}
Consider the sequence $\{\u_k\}_{k=1}^N$, where for each $k$, $\u_k$
is an ideal of $k$-homogeneous polynomials and $N$ is eventually
infinite. We say that $\{\u_k\}_k$ is a \textbf{coherent sequence of
polynomial ideals} if there exist positive constants $C$ and $D$
such that for every Banach
spaces $E$ and $F$, the following conditions hold for $k=1,\dots, N-1$: \\
(i) For each $P\in \u_{k+1}(E,F)$ and $a\in E$, $ P_a$ belongs to
$\u_k(E;F)$ and
$$\|P_a\|_{\u_{k}(E,F)} \le C
\|P\|_{\u_{k+1}(E,F)} \|a\|$$ 
(ii) For each $P\in \u_k(E,F)$ and $\gamma\in E'$, $\gamma P$
belongs to $\u_{k+1}(E,F)$ and
$$\|\gamma P\|_{\u_{k+1}(E,F)}\le D \|\gamma\|
\|P\|_{\u_k(E,F)}$$
\end{definition}

It is shown in \cite{CarDimMur09} that, given an operator ideal $\u$, there are many coherent sequences $\{\u_k\}_k$ such that $\u_1=\u$. On the other hand, an ideal of $n$-homogeneous polynomials $\u_n$ there can be at most one coherent sequence $\{\u_1,\u_2,\dots,\u_n\}$. In other words, all coherent sequences with the same $n$-homogeneous ideal, must have the same $k$-homogeneous ideals for $1\le k\le n$.
As in the case of compatibility, not all the usual polynomial extensions of an
operator ideal form a coherent sequence.
Indeed, an argument similar to the proof of Example \ref{abs sum compatibles con todos} proves:
\begin{example}\rm
The sequence $\{\l,\p^2,\dots,\p^{n-1},\Pi^{n}_p\}$ is coherent with constants $C=e$ and $D=1$.
\end{example}
From this example and \cite[Proposition 1.6]{CarDimMur09} we
can prove that if every absolutely $p$-summing $n$-homogeneous polynomial from $E$ to $F$ is weakly compact, then every $k$-homogeneous polynomial from $E$ to $F$ is weakly compact, for each $k\le n-1$.

\bigskip
For $\u_n$  a Banach ideal of $n$-homogeneous polynomials, we can define for each Banach spaces $E,F$,
$$
\u_{n-1}(E,F) = \left\{ P\in \mathcal \p^{n-1}(E,F) :\; \gamma P\in
\u_{n}(E,F) \textrm{ for all } \gamma\in E' \right\},
$$
with $\tri{P}_{\u_{n-1}(E,F)}=\sup_{\gamma\in
S_{E'}}\|\gamma P\|_{\u_{n}(E,F)}$. With some modifications to the results from previous section (see \cite{Mur10} for details) it may be proven that $\u_{n-1}$ is an ideal of $(n-1)$-homogenous polynomials and that $\tri{\cdot}_{\u_{n-1}}$ is an almost ideal norm, which can be modified to an equivalent ideal norm $\|\cdot\| _{\u_{n}(E,F)}$.
We can proceed analogously to define $\u_{n-2}$, and then $\u_{n-3},\dots,\u_1$. As a result, we have shown how to construct a (necessarily  unique) sequence of  polynomial ideals $\u_1,\dots,\u_{n-1}$ such that $\{\u_1,\dots,\u_{n-1},\u_{n}\}$ is a coherent sequence. Also, the polynomial ideals $\u_1,\dots,\u_{n-1}$ can be normed to obtain constants of coherence $1\le C,D\le e$.

We can also give a whole sequence $\{\u_k\}_{k=1}^\infty$ of polynomial ideals which form a coherent sequence. In this case the ideals $\u_k$, $k\ge n+1$ are not uniquely determined by $\u_n$. For example, we may define for each Banach spaces $E,F$, $k\ge1$,
$$
\u_{n+k}(E,F)=\{P\in\p^{n+k}(E,F):\, P_{a^k}\textrm{ belongs to }\u_n(E,F)\textrm{ for every }a\in E\},
$$
and $\|P\|_{\u_{n+k}(E,F)}=\sup\{\|P_{a^k}\|_{\u_n(E,F)}:\, \|a\|_E=1\}$. Then it is easy to see that $\u_{n+k}$ is an ideal of $(n+k)$-homogeneous polynomial and that $\|\cdot\|_{\u_{n+k}}$ is an ideal norm. Moreover, a simple modification of the proof of \cite[Proposition 2.5 $a$)]{CarDimMur09} shows that if $P\in\u_{n+k}(E,F)$ and $a\in E$ then $P_a$ belongs to $\u_{n+k-1}(E,F)$ and $\|P_a\|_{\u_{n+k-1}(E,F)}\le e\|a\|\|P\|_{\u_{n+k}(E,F)}$. Therefore we have the following.
\begin{theorem} \label{existecoherente}
Let $\u_{n}$ be a Banach ideal of $n$-homogeneous polynomials. Then
there exist polynomial ideals $\u_1,\dots,\u_{n-1},\u_{n+1},\dots$ such that $\{\u_k\}_{k=1}^\infty$ is a coherent sequence with constants $1\le C,D\le e$. The polynomial ideals $\u_1,\dots,\u_{n-1}$ are uniquely determined by $\u_{n}$.
\end{theorem}

  The sequence $\{\u_k\}_{k=1}^\infty$ constructed is actually the largest sequence of ideals coherent with $\u_n$. That is, if $\{\mathfrak B_k\}_{k=1}^\infty$ is a coherent sequence such that $\mathfrak B_n=\u_n$, then $\mathfrak B_k(E,F)\subset \u_k(E,F)$ for every $E,F$,
$k\in\mathbb N$ (it is an equality for $k<n$).
It is also possible to define the smallest coherent sequence associated to $\u_n$, see
\cite[Section 2]{CarDimMur09} for a related construction.

It is clear that we can deduce the existence of compatible operator
ideal for every polynomial ideal from Theorem~\ref{existecoherente}.
However, is should be noted that the bounds for the compatibility
constants obtained with this theorem would be $e^n$, while in
Theorem~\ref{existcompat} we have $e$ as a bound.

%\bibliography{biblio}
%\bibliographystyle{amsplain}
\providecommand{\bysame}{\leavevmode\hbox
to3em{\hrulefill}\thinspace}
\providecommand{\MR}{\relax\ifhmode\unskip\space\fi MR }
\providecommand{\MRhref}[2]{%
  \href{http://www.ams.org/mathscinet-getitem?mr=#1}{#2}
} \providecommand{\href}[2]{#2}

\end{document}